\documentclass{amsart}

\usepackage[margin=1.2in]{geometry}

\usepackage{amsthm}
\usepackage{amsmath}
\usepackage{amssymb}
\usepackage{enumerate}
\usepackage{graphicx}
\usepackage[pagebackref,pdftex]{hyperref}
\usepackage{url}
\usepackage{color}
\usepackage[dvipsnames]{xcolor}
\usepackage{import}
\usepackage{tikz-cd}
\usepackage{overpic}

\AtBeginDocument{%
  \def\MR#1{}
}

\usepackage{marginnote}
\long\def\@savemarbox#1#2{\global\setbox#1\vtop{\hsize\marginparwidth 
  \@parboxrestore\tiny\raggedright #2}}
\numberwithin{equation}{section}
\theoremstyle{plain}
\newtheorem{theorem}[equation]{Theorem}

\newtheorem{lemma}[equation]{Lemma}
\newtheorem{conjecture}[equation]{Conjecture}

\newtheorem*{namedtheorem}{\theoremname}
\newcommand{\theoremname}{testing}

\theoremstyle{definition}
\newtheorem{definition}[equation]{Definition}

\newtheorem{question}[equation]{Question}

\newcommand{\refthm}[1]{Theorem~\ref{Thm:#1}}
\newcommand{\reflem}[1]{Lemma~\ref{Lem:#1}}

\newcommand{\refconj}[1]{Conjecture~\ref{Conj:#1}}
\newcommand{\refeqn}[1]{\eqref{Eqn:#1}}

\newcommand{\refdef}[1]{Definition~\ref{Def:#1}}
\newcommand{\refsec}[1]{Section~\ref{Sec:#1}}
\newcommand{\refques}[1]{Question~\ref{Ques:#1}}
\newcommand{\reffig}[1]{Figure~\ref{Fig:#1}}

\newcommand{\HH}{{\mathbb{H}}}
\newcommand{\RR}{{\mathbb{R}}}
\newcommand{\ZZ}{{\mathbb{Z}}}

\newcommand{\CC}{{\mathbb{C}}}

\newcommand{\from}{\colon} 

\newcommand{\bdy}{\partial}

\renewcommand{\setminus}{\smallsetminus}

\newcommand{\vol}{\operatorname{vol}}
\newcommand{\PSL}{\operatorname{PSL}}

\newcommand{\lmin}{{\ell_{\rm min}}}

\newcommand{\area}{\operatorname{area}}

\newcommand{\abs}[1]{\left\vert #1 \right\vert}
\newcommand{\cut}{{\backslash \backslash}}
\newcommand{\guts}{\operatorname{guts}}

\newcommand{\vtet}{{v_{\rm tet}}}
\newcommand{\voct}{{v_{\rm oct}}}

\title{A survey of hyperbolic knot theory}

\author{David Futer}
\author{Efstratia Kalfagianni}
\author{Jessica S.~Purcell}
\address[]{Department of Mathematics, Temple University,
Philadelphia, PA 19122, USA}

\email[]{dfuter@temple.edu}

\address[]{Department of Mathematics, Michigan State University, East
Lansing, MI, 48824, USA}

\email[]{kalfagia@math.msu.edu}

\address[]{School of Mathematical Sciences, Monash University, VIC 3800, Australia }

\email[]{jessica.purcell@monash.edu}


\begin{document}
\begin{abstract}
We survey some tools and techniques  for determining geometric properties of a link complement from a link diagram. 
In particular, we survey the tools used to estimate geometric invariants in terms of basic diagrammatic link invariants. 
We focus on determining when a link is hyperbolic, estimating its volume, and bounding its cusp shape and cusp area. We give sample applications and state some open questions and conjectures.
\end{abstract}

\maketitle

\section{Introduction}\label{Sec:Intro}

Every link $L\subset S^3$ defines a compact, orientable 3-manifold boundary consisting of tori; namely, the link exterior $X(L)=S^3\setminus N(L)$, where $N(L)$ denotes an open regular neighborhood. The interior of $X(L)$ is homeomorphic to the link complement $S^3\setminus L$. Around 1980, Thurston proved that link complements decompose into pieces that admit locally homogeneous geometric structures. In the most common and most interesting scenario, the entire link complement has a hyperbolic structure, that is a metric of constant curvature $-1$. By Mostow--Prasad rigidity, this hyperbolic structure is unique up to isometry, hence geometric invariants of $S^3\setminus L$ give topological invariants of $L$ that provide a wealth of information about $L$ to aid in its classification.  

An important and difficult problem is to determine the geometry of a link complement directly from link diagrams, and to estimate geometric invariants such as volume and the lengths of geodesics in terms of basic diagrammatic  invariants of $L$. This problem often goes by the names \emph{WYSIWYG topology}{\footnote{\emph{WYSIWYG}  stands for ``what you see is what you get''.}}
or \emph{effective geometrization} \cite{Johnson:WYSIWYG}.  Our purpose in this paper is to survey some results that effectively predict geometry in terms of diagrams, and to state some open questions. In the process, we also summarize some of the most commonly used tools and techniques that have been employed to study this problem.

\subsection{Scope and aims}
This survey is primarily devoted to three main topics: determining when a knot or link is hyperbolic, bounding its volume, and estimating its cusp geometry. Our main goal is to focus 
on the methods, techniques, and tools of the field, in the hopes that this paper will lead to more research, rather than strictly listing previous results. 

This focus overlaps significantly with the list of topics in Adams' survey article \emph{Hyperbolic knots}~\cite{Adams:survey}. That survey, written in 2003 and published in 2005, came out just as the pursuit of effective geometrization was starting to mature. Thus, although the topics are quite similar, both the results and the underlying techniques have advanced to a considerable extent. This is especially visible in efforts to predict hyperbolic volume (\refsec{Volume}), where only a handful of the results that we list were known by 2003. The same pattern asserts itself throughout.

As with all survey articles, the list of results  and open problems that we can address is necessarily incomplete.
We are not addressing the very interesting questions on the geometry of embedded surfaces, lengths and isotopy classes of geodesics, exceptional Dehn fillings, or geometric properties of other knot and link invariants. 
Some of the results and techniques we have been unable to cover will appear in a forthcoming book in preparation by Purcell \cite{Purcell:book}. 

\subsection{Originality, or lack thereof}
With one exception, all of the results presented in this survey have appeared elsewhere in the literature. For all of these results, we point to references rather than giving rigorous proofs. However, we often include quick sketches of arguments to convey a sense of the methods that have been employed.

The one exception to this rule is \refthm{SymmetricKnotVolume}, which has not previously appeared in writing. Even this result cannot be described as truly original, since the proof works by assembling a number of published theorems. We include the proof to indicate how to assemble the ingredients.

\subsection{Organization}
We organize this survey as follows: \refsec{Definitions} introduces terminology and background that we will use throughout. \refsec{Hyperbolic} is concerned with the problem of determining whether a given link is hyperbolic. We summarize some of the most commonly used methods used for this problem, and provide examples. In \refsec{Volume} and \refsec{Cusps}, we address the problem of estimating important geometric invariants of hyperbolic link complements in terms of  diagrammatic quantities. In \refsec{Volume}, we discuss methods for obtaining two sided combinatorial bounds on the hyperbolic volume of link complements. In \refsec{Cusps}, we address the analogous questions for cusp shapes and for lengths of curves on cusp tori.

\subsection{Acknowledgements} Futer is supported in part by NSF grant DMS--1408682.
Kalfagianni is supported in part by NSF grants DMS-1404754 and DMS-1708249. Purcell is supported in part by the Australian Research Council.
All three authors acknowledge support from NSF grants DMS--1107452, 1107263, 1107367, ``RNMS: Geometric Structures and Representation Varieties'' (the GEAR Network).

\section{Definitions}\label{Sec:Definitions}

In this section, we gather many of the key definitions that will be used throughout the paper. Most of these definitions can be found (and are better motivated) in standard textbooks on knots and links, and on $3$--manifolds and hyperbolic geometry. We list them briefly for ease of reference.

\subsection{Diagrams of knots and links}
Some of the initial study of knots and links, such as the work of Tait in the late 1800s, was a study of \emph{diagrams}: projections of a knot or link onto a plane $\RR^2 \subset \RR^3$, which can be compactified to $S^2\subset S^3$. We call the surface of projection the \emph{plane of projection} for the diagram. We may assume that a link has a diagram that is a 4-valent graph on $S^2$, with over-under crossing information at each vertex. When studying a knot via diagrams, there are obvious moves that one can make to the diagram that do not affect the equivalence class of knot; for example these include \emph{flypes} studied by Tait, and Reidemeister moves studied in the 1930s. Without going into details on these moves, we do want our diagrams to be ``sufficiently reduced,'' in ways that are indicated by the following definitions. 

\begin{definition}\label{Def:Prime}
A diagram of a link is \emph{prime} if for any simple closed curve $\gamma \subset S^2$, intersecting the diagram transversely exactly twice in edges, $\gamma$ bounds a disk $D^2 \subset S^2$ that intersects the diagram in a single edge with no crossings. 
\end{definition}

Two non-prime diagrams are shown in \reffig{Prime}, left. The first diagram can be simplified by removing a crossing. The second diagram cannot be reduced in the same way, because the knot is composite; it can be thought of as composed of two simpler prime diagrams by joining them along unknotted arcs. Prime diagrams are seen as building blocks of all knots and links, and so we restrict to them.

\begin{figure}
   \includegraphics{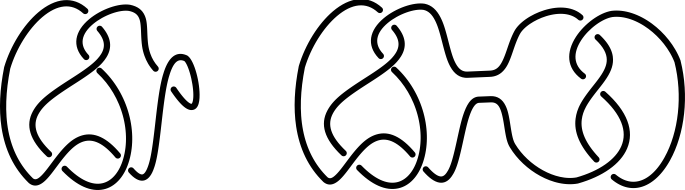}
  \hspace{.5in}
    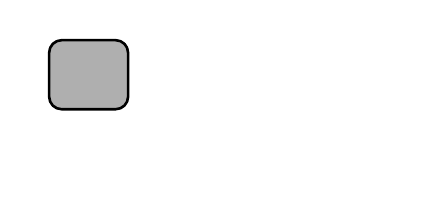
  \caption{Left: two diagrams that are not prime. Right: a twist reduced diagram.}
  \label{Fig:Prime}
\end{figure}

\begin{definition}\label{Def:CrossingNumber}
  Suppose $K$ is a knot or link with diagram $D$. The \emph{crossing number} of the diagram, denoted $c(D)$, is the number of crossings in $D$. 
The \emph{crossing number} of $K$, denoted $c(K)$, is defined to be the minimal number of crossings in any diagram of $K$.
\end{definition}

Removing a crossing as on the left of \reffig{Prime} gives a diagram that is more reduced. The following definition gives another way to reduce diagrams.

\begin{definition}\label{Def:TwistReduced}
Let $K$ be a knot or link with diagram $D$. The diagram is said to be \emph{twist reduced} if whenever $\gamma$ is a simple closed curve in the plane of projection meeting the diagram exactly twice in two crossings, running directly through the crossing, then $\gamma$ bounds a disk containing only a string of alternating bigon regions in the diagram. See \reffig{Prime}, right. 
\end{definition}

Any diagram can be modified to be twist reduced by performing a sequence of flypes and removing unnecessary crossings. 

\begin{definition}\label{Def:TwistNumber}
  A \emph{twist region} in a diagram is a portion of the diagram consisting of a maximal string of bigons arranged end-to-end, where maximal means there are no other bigons adjacent to the ends. Additionally, we require twist regions to be alternating (otherwise, remove crossings).

  The number of twist regions in a prime, twist reduced diagram is the \emph{twist number} of the diagram, and is denoted $t(D)$. The minimum of $t(D)$ over all diagrams of $K$ is denoted $t(K)$.
\end{definition}

\subsection{The link complement}

Rather than study knots exclusively via diagrams and graphs, we typically consider the \emph{knot complement}, namely the $3$--manifold $S^3\setminus K$. This is homeomorphic to the interior of the compact manifold $X(K):=S^3\setminus N(K)$, called the \emph{knot exterior}, where $N(K)$ is a regular neighborhood of $K$. When we consider knot complements and knot exteriors, we are able to apply results in $3$--manifold topology, and consider curves and surfaces embedded in them. The following definitions apply to such surfaces.

\begin{definition}\label{Def:Incompressible}
  An orientable surface $S$ properly embedded in a compact orientable $3$--manifold $\overline{M}$ is \emph{incompressible} if whenever $E\subset \overline{M}$ is a disk with $\bdy E\subset S$, there exists a disk $E'\subset S$ with $\bdy E' = \bdy E$.
  $S$ is \emph{$\bdy$-incompressible} if whenever $E\subset \overline{M}$ is a disk whose boundary is made up of an arc $\alpha$ on $S$ and an arc on $\bdy \overline{M}$, there exists a disk $E'\subset S$ whose boundary is made up of the arc $\alpha$ on $S$ and an arc on $\bdy S$.
\end{definition}

\begin{definition}\label{Def:Essential}
Let $\overline{M}$ be a compact orientable $3$--manifold. 
Consider a (possibly non-orientable) properly embedded surface $S \subset \overline{M}$. Let $\widetilde{S}$ be the boundary of a regular neighborhood $N(S) \subset \overline{M}$. If $S \neq S^2$, it  is said to be \emph{essential} if $\widetilde{S}$ is incompressible and $\bdy$-incompressible. A two--sphere $S \subset M$ is called \emph{essential} if it does not bound a $3$--ball.
  
  We will say that $\overline{M}$ is  \emph{Haken} if it is irreducible and contains an essential surface $S$. In this case, we also say the interior $M$ is Haken.
\end{definition}

Finally, we will sometimes consider knot complements that are fibered, in the sense of the following definition.

\begin{definition}\label{Def:Fibered}
  A $3$--manifold $M$ is said to be \emph{fibered} if it can be written as a fiber bundle over $S^1$, with fiber a surface. Equivalently, $M$ is the mapping torus of a self-homeomorphism $f$ of a (possibly punctured) surface $S$. That is, there exists $f\from S\to S$ such that
  \[ M = S\times I / (x,0) \sim (f(x),1). \]
  The map $f$ is called the \emph{monodromy} of the fibration. 
\end{definition}

\subsection{Hyperbolic geometry notions}

The knot and link complements that we address in this article also admit geometric structures, as in the following definition. 

\begin{definition}\label{Def:HyperbolicKnot}
  A knot or link $K$ is said to be \emph{hyperbolic} if its complement admits a complete metric of constant curvature $-1$. Equivalently, it is hyperbolic 
if $S^3 \setminus K = \HH^3 / \Gamma$, where $\HH^3$ is hyperbolic $3$--space and $\Gamma$ is a discrete, torsion-free group of isometries, isomorphic to $\pi_1(S^3 \setminus K)$.
\end{definition}

Thurston showed that a prime knot in $S^3$ is either hyperbolic, or it is a \emph{torus knot} (can be embedded on an unknotted torus in $S^3$), or it is a \emph{satellite knot} (can be embedded in the regular neighborhood of a non-trivial knot) \cite{thurston:bulletin}. This article is concerned with hyperbolic knots and links. 

\begin{definition}\label{Def:CuspedManifold}
  Suppose $\overline{M}$ is a compact orientable $3$--manifold with $\bdy M$ a collection of tori, and suppose the interior $M \subset \overline{M}$ admits a complete hyperbolic structure. We say $M$ is a \emph{cusped manifold}. 

  Moreover, $M$ has ends of the form $T^2\times [1,\infty)$. Under the covering map $\rho\from \HH^3\to M$, each end is geometrically realized as the image of a horoball $H_i\subset \HH^3$. The preimage $\rho^{-1}(\rho(H_i))$ is a collection of horoballs. By shrinking $H_i$ if necessary, we can ensure that these horoballs have disjoint interiors in $\HH^3$. For such a choice of $H_i$, $\rho(H_i) = C_i$ is said to be a \emph{horoball neighborhood} of the \emph{cusp} $C_i$, or \emph{horocusp} in $M$.
\end{definition}

\begin{definition}\label{Def:CuspShape}
  The boundary of a horocusp inherits a Euclidean structure from the hyperbolic structure on $M$. This Euclidean structure is well defined up to similarity. The similarity class is called the \emph{cusp shape}.
\end{definition}

\begin{definition}\label{Def:MaximalCusp}
  For each cusp of $M$ there is an 1--parameter family of horoball neighborhoods obtained by expanding the horoball $H_i$ while keeping the same limiting point on the sphere at infinity. In the preimage, expanding $H_i$ expands all horoballs in the collection $\rho^{-1}(C_i)$. Expand each cusp until the collection of horoballs $\rho^{-1}(\cup C_i)$ become tangent, and cannot be expanded further while keeping their interiors disjoint. This is a choice of \emph{maximal cusps}. The choice depends on the order of expansion of cusps $C_1, \dots, C_n$. If $M$ has a single end $C_1$ then there is a unique choice of expansion, giving a unique maximal cusp referred to as the \emph{the maximal cusp} of $M$.
\end{definition}

\begin{definition}\label{Def:Slope}
For a fixed set of embedded horoball neighborhoods $C_1, \dots, C_n$ of the cusps of a cusped hyperbolic $3$--manifold $M$, we have noted that the torus $\bdy C_i$ inherits a Euclidean metric. Any isotopy class of simple closed curves on the torus is called a \emph{slope}. The \emph{length of a slope} $s$, denoted $\ell(s)$, is defined to be the length of a geodesic representative of $s$ on the Euclidean torus $\bdy C_i$.
\end{definition}

\section{Determining hyperbolicity}\label{Sec:Hyperbolic}

Given a combinatorial description of a knot or link, such as a diagram or braid presentation, one of the first things we would often like to ascertain is whether  the link complement admits a hyperbolic structure at all. In this section, we describe the currently available tools to check this and give examples of knots to which they apply. 

There are three main tools used to prove a link or family of links is hyperbolic. The first is direct calculation, for example using gluing and completeness equations, often with the help of a computer. The second is Thurston's geometrization theorem for Haken manifolds, which says that the only obstruction to $X(K)$ being hyperbolic consists of surfaces with non-negative Euler characteristic. The third is to perform a long Dehn filling on a manifold that is already known to be hyperbolic, for instance by one of the previous two methods. 

\subsection{Computing hyperbolicity directly}

From Riemannian geometry, a manifold $M$ admits a hyperbolic structure if and only if $M = \HH^3 / \Gamma$, where $\Gamma \cong \pi_1(M)$ is a discrete subgroup of $\mathrm{Isom}^+(\HH^3) = \PSL(2,\CC)$. See \refdef{HyperbolicKnot}.

 Therefore one way to find a hyperbolic structure on a link complement is to find a discrete faithful representation of its fundamental group into $\PSL(2,\CC)$. This is usually impractical to do directly. However, note that if a manifold $M$ can be decomposed into simply connected pieces, for example a triangulation by tetrahedra, then these lift to the universal cover. If this cover is isometric to $ \HH^3$, then the lifted tetrahedra will be well-behaved in hyperbolic 3--space. Conversely, if the lifted tetrahedra fit together coherently in $\HH^3$, in a group--equivariant fashion, one can glue the metrics on those tetrahedra to obtain a hyperbolic metric on $M$. This gives a condition for determining hyperbolicity, which is often implemented in practice.

\subsubsection{Gluing and completeness equations for triangulations}
The first method for finding a hyperbolic structure is direct, and is used most frequently by computer, such as in the software SnapPy that computes hyperbolic structures directly from  diagrams \cite{SnapPy}. The method is to first decompose the knot or link complement into ideal tetrahedra, as in \refdef{IdealTetrahedron}, and then to solve a system of equations on the tetrahedra to obtain a hyperbolic structure. See \refthm{GluingCompleteness}.

This method is most useful for a single example, or for a finite collection of examples. For example, it was used by Hoste, Thistlethwaite, and Weeks to classify all prime knots with up to 16 crossings \cite{HTW}. Of the $1,701,903$ distinct prime knots with at most 16 crossings, all but  32 are hyperbolic.

We will give a brief description of the method. For further details, there are several good references, including notes of Thurston \cite{thurston:notes} where these ideas first appeared, and papers by Neumann and Zagier~\cite{NeumannZagier}, and Futer and Gu{\'e}ritaud~\cite{fg:angled-survey}. Purcell is developing a book with details and examples~\cite{Purcell:book}.

\begin{definition}\label{Def:IdealTetrahedron}
  An \emph{ideal tetrahedron} is a tetrahedron whose vertices have been removed. When a knot or link complement is decomposed into ideal tetrahedra, all ideal vertices lie on the link, hence have been removed.
\end{definition}

There are algorithms for decomposing knot and link complements into ideal tetrahedra. For example, Thurston decomposes the figure--8 knot complement into two ideal tetrahedra \cite{thurston:notes}. Menasco generalizes this, describing how to decompose a link complement into two ideal polyhedra, which can then be subdivided into tetrahedra \cite{Menasco:Polyhedra}. Weeks uses a different algorithm in his computer software SnapPea \cite{Weeks:Algorithm}.

Assuming we have a decomposition of a knot or link complement into ideal tetrahedra, we now describe how to turn this into a complete hyperbolic structure. The idea is to associate a complex number to each ideal edge of each tetrahedron encoding the hyperbolic structure of the ideal tetrahedron. The triangulation gives a complete hyperbolic structure if and only if these complex numbers satisfy certain equations: the \emph{edge gluing} and \emph{completeness} equations.

Consider $\HH^3$ in the upper half space model, $\HH^3 \cong \CC\times (0, \infty)$. An ideal tetrahedron $\Delta \subset \HH^3$ can be moved by isometry so that three of its vertices are placed at $0$, $1$, and $\infty$ in $\bdy\HH^3 \cong \CC\cup \{\infty\}$. The fourth vertex lies at a point $z\in \CC \setminus \{0, 1 \}$. The edges between these vertices are hyperbolic geodesics.

\begin{definition}
  The parameter $z\in\CC$ described above is called the \emph{edge parameter} associated with the edge from $0$ to $\infty$. It determines $\Delta$ up to isometry.
\end{definition}

Notice if $z$ is real, then the ideal tetrahedron is flat, with no volume. We will prefer to work with $z$ with positive imaginary part. Such a tetrahedron $\Delta$ is said to be \emph{geometric}, or positively oriented. If $z$ has negative imaginary part, the tetrahedron $\Delta$ is negatively oriented. 

Given a hyperbolic ideal tetrahedron embedded in $\HH^3$ as above, we can apply (orientation--preserving) isometries of $\HH^3$ taking different vertices to $0$, $1$, $\infty$.
By taking each edge to the geodesic from $0$ to $\infty$, we assign edge parameters to all six edges of the ideal tetrahedron. This leads to the following relations between edge parameters:

\begin{lemma}
  Suppose $\Delta$ is a hyperbolic ideal tetrahedron with vertices at $0$, $1$, $\infty$, and $z$. Then the edge parameters of the six edges of $\Delta$ are as follows:
 \begin{itemize}
 \item Edges $[0, \infty]$ and $[1 , z]$ have edge parameter $z$.
 \item Edges $[1 , \infty]$ and $[0 , z]$ have edge parameter $1/(1-z)$.
 \item Edges $[z , \infty]$ and $[0 , 1]$ have edge parameter $(z-1)/z$.
 \end{itemize} 
In particular, opposite edges in the tetrahedron have the same edge parameter.
\end{lemma}

Suppose an ideal tetrahedron $\Delta$ with vertices at $0$, $1$, $\infty$ and $z$ is glued along the triangle face with vertices at $0$, $\infty$, and $z$ to another tetrahedron $\Delta'$. Then $\Delta'$ will have vertices at $0$, $\infty$, $z$ and at the point $zw$, where $w$ is the edge parameter of $\Delta'$ along the edge  $[0, \infty]$. When we glue all tetrahedra in $\HH^3$ around an ideal edge of the triangulation, if the result is hyperbolic then the product of all edge parameters must be $1$ with arguments summing to $2\pi$. More precisely, the sum of the logs of the edge parameters must be $0+2\pi\,i$. 

\begin{definition}[Gluing equations]
  Let $e$ be an ideal edge of a triangulation of a $3$--manifold $M$, for example a knot or link complement. Let $z_1, \dots, z_k$ be the edge parameters of the edge of the tetrahedra identified to $e$. The \emph{gluing equation} associated with the edge $e$ is:
\begin{equation}\label{Eqn:Gluing}
 \prod_{i=1}^k z_i = 1 \quad \mbox{and} \quad \sum_{i=1}^k \arg(z_i)=2\pi.
 \end{equation}
  Writing this in terms of logarithms, \refeqn{Gluing} is equivalent to:
  \begin{equation}\label{Eqn:LogGluing}
   \sum_{i=1}^k \log(z_i) = 2\pi\,i. 
   \end{equation}
\end{definition}

A triangulation may satisfy all gluing equations at all its edges, and yet fail to give a complete hyperbolic structure. To ensure the structure is complete, an additional condition must be satisfied for each torus boundary component. 

\begin{definition}[Completeness equations]
  Let $T$ be a torus boundary component of a $3$--manifold $M$ whose interior admits an ideal triangulation. 
 
  Truncate the tips of all tetrahedra to obtain a triangulation of $T$. Let $\mu$ be an oriented simple closed curve on $T$, isotoped to meet edges of the triangulation transversely, and to avoid vertices. Each segment of $\mu$ in a triangle cuts off a single vertex of the triangle, which comes from an edge of the ideal triangulation and so has an associated edge parameter $z_i$. If the vertex lies to the right of  $\mu$, let $\epsilon_i=+1$; otherwise let $\epsilon_i=-1$. The \emph{completeness equation} associated to $\mu$ is:
  \begin{equation}
\sum_i \epsilon_i\,\log(z_i) = 0,
\quad  \mbox{ which implies } \quad
   \prod_i z_i^{\epsilon_i} =1.
\end{equation}
\end{definition}

With these definitions, we may state the main theorem.

\begin{theorem}\label{Thm:GluingCompleteness}
  Suppose $\overline{M}$ is a $3$-manifold with torus boundary, equipped with an ideal triangulation. Suppose for some choice of positively oriented edge parameters $\{ z_1, \dots, z_n\}$, the gluing equations are satisfied for each edge, and the completeness equations are satisfied for homology generators $\mu$, $\lambda$ on each component of $\bdy {\overline{M}}$. Then the interior of 
  $\overline{M}$, denoted by $M$, admits a complete hyperbolic structure. Furthermore, the unique hyperbolic metric on $M$ is given by the geometric tetrahedra  determined by the edge parameters. 
\end{theorem} 

In fact, the hypotheses of \refthm{GluingCompleteness} are stronger than necessary. If $\overline{M}$ has $k$ torus boundary components, then only $n-k$ of the $n$ gluing equations are necessary (see  \cite{NeumannZagier} or \cite{Choi}). In addition, only one of $\mu$ or $\lambda$ is required from each boundary component \cite{Choi}. 

Some classes of $3$--manifolds that can be shown to be hyperbolic using \refthm{GluingCompleteness} include the classes of once-punctured torus bundles, 4-punctured sphere bundles, and 2--bridge link complements \cite{GueritaudFuter:2bridge}. (In each class, some low-complexity examples must be excluded to ensure hyperbolicity.) These manifolds have natural ideal triangulations guided by combinatorics. In the case of $2$--bridge knot and link complements, the triangulation is also naturally adapted to a planar diagram of the link \cite{SakumaWeeks}. Once certain low-complexity cases (such as ($2,q$) torus links) have been excluded, one can show that the gluing equations for these triangulations have a solution. This gives a direct proof that the manifolds are hyperbolic. 

\subsubsection{Circle packings and right angled polyhedra}\label{Sec:CirclePack}
Certain link complements have very special geometric properties that allow us to compute their hyperbolic structure directly, but with less work than solving nonlinear gluing and completeness equations as above. These include the Whitehead link, which can be obtained from a regular ideal octahedron with face-identifications \cite{thurston:notes}. They also include an important and fairly general family of link complements called \emph{fully augmented links}, which we now describe. 

Starting with any knot or link diagram, identify \emph{twist regions}, as in \refdef{TwistNumber}. The left of \reffig{Augmented} shows a knot diagram with two twist regions. Now, to each twist region, add a simple unknotted closed curve encircling the two strands of the twist region, as shown in the middle of \reffig{Augmented}. This is called a \emph{crossing circle}. Because each crossing circle is an unknot, we may perform a full twist along a disk bounded by that unknot without changing the homeomorphism type of the link complement.

This allows us to remove as many pairs of crossings as possible from twist regions. An example is shown on the right of \reffig{Augmented}. The result is the diagram of a fully augmented link. 

\begin{figure}
\includegraphics{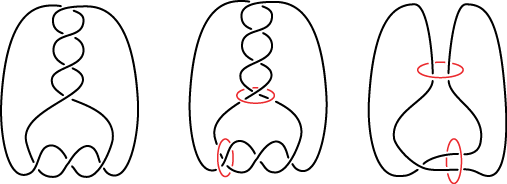}
\caption{Left: a diagram of a knot $K$. Center: adding a crossing circles for each twist region of $K$ produces a link $J$. Right: removing full twists produces a \emph{fully augmented link} $L$ with the property that $S^3 \setminus J$ is homeomorphic to $S^3 \setminus L$.}
\label{Fig:Augmented}
\end{figure}

Provided the original link diagram before adding crossing circles is sufficiently reduced (prime and twist reduced; see Definitions~\ref{Def:Prime} and~\ref{Def:TwistReduced}), the resulting fully augmented link will be hyperbolic, and its hyperbolic structure can be completely determined by a circle packing. The procedure is as follows.

Replace the diagram of the fully augmented link with a trivalent graph by replacing each neighborhood of a crossing circle (with or without a bounded crossing) by a single edge running between knot strands, closing the knot strands. See \reffig{Decomp}, left. Now take the dual of this trivalent graph; this is a triangulation of $S^2$. Provided the original diagram was reduced, there will be a circle packing whose nerve is this triangulation of $S^2$. The circle packing and its orthogonal circles cut out a right angled ideal polyhedron in $\HH^3$. The hyperbolic structure on the complement of the fully augmented link is obtained by gluing two copies of this right angled ideal polyhedron. More details are in \cite{FuterPurcell, purcell:augmented}.

\begin{figure}
  \includegraphics{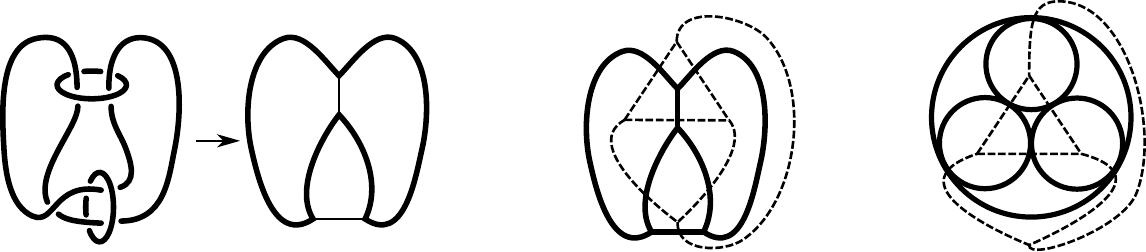}
  \caption{Left: Obtain a 3-valent graph by replacing crossing circles with edges. Middle: The dual is a triangulation of $S^2$. Right: The nerve of the triangulation defines a circle packing that cuts out a polyhedron in $\HH^3$. Two such polyhedra glue to form $S^3\setminus L$.}
  \label{Fig:Decomp}
\end{figure}

\subsection{Geometrization of Haken manifolds}

The methods of the previous section have several drawbacks. While solving gluing and completeness equations works well for examples, it is difficult to use the methods to find hyperbolic structures for infinite classes of examples. The method that has been most useful to show infinite examples of knots and links are hyperbolic is to apply Thurston's geometrization theorem for Haken manifolds, which takes the following form for manifolds with torus boundary components.

\begin{theorem}[Geometrization of Haken manifolds]\label{Thm:Haken}
  Let $M$ be the interior of a compact manifold $\overline{M}$, such that $\bdy \overline{M}$ is a non-empty union of tori. Then exactly one of the following holds:
  \begin{itemize}
  \item $\overline{M}$ admits an essential torus, annulus, sphere, or disk, or
  \item $M$ admits a complete hyperbolic metric.
  \end{itemize}
 \end{theorem}

Thus the method to prove $M$ is hyperbolic following \refthm{Haken} is to show
$\overline{M}$ cannot admit embedded essential surfaces of nonnegative Euler characteristic. Arguments ruling out such surfaces are typically topological or combinatorial in nature.

Some sample applications of this method are as follows. Menasco used the method to prove any alternating knot or link, aside from a $(2,q)$-torus link, is hyperbolic \cite{Menasco:Alternating}. Adams and his students generalized Menasco's argument to show that almost alternating and toroidally alternating links are hyperbolic \cite{Adams:ToroidallyAlt, Adams:AlmostAlt}. There are many other generalizations, e.g.\ \cite{fkp:hyp}. 

Menasco's idea was to subdivide an alternating link complement into two balls, above and below the plane of projection, and crossing balls lying  in a small neighborhood of each crossing, with equator along the plane of projection. An essential surface can be shown to intersect the balls above and below the plane of projection in disks only, and to intersect crossing balls in what are called \emph{saddles}. These saddles act as fat vertices on the surface, and can be used to obtain a bound on the Euler characteristic of an embedded essential surface. Combinatorial arguments, using properties of alternating diagrams, then rule out surfaces with non-negative Euler characteristic.

More generally, classes of knots and links can be subdivided into simpler pieces, whose intersection with essential surfaces is then examined. Typically, surfaces with nonnegative Euler characteristic can be restricted to lie in just one or two pieces, and then eliminated.
\medskip

Thurston's \refthm{Haken} can also be used to show that manifolds with certain properties are hyperbolic. For example, consider again the gluing equations. This gives a complicated nonlinear system of equations. If we consider only the imaginary part of the logarithmic gluing equation \refeqn{LogGluing},
the system becomes linear: the sums of dihedral angles around each edge must be $2\pi$. It is much easier to solve such a system of equations.

\begin{definition}\label{Def:AngleStruct}
  Suppose $M$ is the interior of a compact manifold with torus boundary, with an ideal triangulation. A solution to the imaginary part of the (logarithmic) gluing equations \refeqn{LogGluing} for the triangulation is called a \emph{generalized angle structure} on $M$. If all angles lie strictly between $0$ and $\pi$, the solution is called an \emph{angle structure}. See \cite{fg:angled-survey, luo-tillmann:generalized-angles} for background on (generalized) angle structures.
\end{definition}

\begin{theorem}[Angle structures and hyperbolicity]\label{Thm:AngleStruct}
  If $M$ admits an angle structure, then $M$ also admits a hyperbolic metric. 
\end{theorem}

The proof has been attributed to Casson, and appears in Lackenby  \cite{Lackenby:Word}. The idea is to consider how essential surfaces intersect each tetrahedron of the triangulation. These surfaces can be isotoped into \emph{normal form}. A surface without boundary in normal form intersects tetrahedra only in triangles and in quads. The angle structure on $M$ can be used to define a combinatorial area on a normal surface. An adaptation of the Gauss--Bonnet theorem implies that the Euler characteristic is a negative multiple of  the combinatorial area. Then one shows that the combinatorial area of an essential surface must always be strictly positive, hence Euler characteristic is strictly negative. Then \refthm{Haken} gives the result.
\medskip

Knots and links that can be shown to be hyperbolic using the tools of \refthm{AngleStruct} include arborescent links, apart from three enumerated families of non-hyperbolic exceptions. This can be shown by constructing an ideal triangulation (or a slightly more general ideal decomposition) of the complement of an arborescent link, and endowing it with an angle structure \cite{FG:Arborescent}.

Conversely, every hyperbolic knot or link complement in $S^3$ admits \emph{some} ideal triangulation with an angle structure \cite{HRS:AnglesHomology}. However, this triangulation is not explicitly constructed, and need not have any relation to the combinatorics of a diagram.

\subsection{Hyperbolic Dehn filling}

Another method for proving that classes of knots or links are hyperbolic is to use Dehn filling. Thurston showed that all but finitely many Dehn fillings on a hyperbolic manifold with a single cusp yield a closed hyperbolic $3$--manifold \cite{thurston:notes}. 

More effective versions of Thurston's theorem have been exploited to show hyperbolicity all but a bounded number of Dehn fillings.  Results in this vein include $2\pi$--theorem that yields negatively curved metrics \cite{BleilerHodgson:2pi}, and geometric deformation theorems of Hodgson and Kerckhoff \cite{HK:univ}. The sharpest result along these lines is the $6$--Theorem, due independently to Agol \cite{Agol:Bounds} and Lackenby \cite{Lackenby:Word}. (The statement below assumes the geometrization conjecture, proved by Perelman shortly after the papers \cite{Agol:Bounds, Lackenby:Word} were published.)

\begin{theorem}[6--Theorem]\label{Thm:6Theorem}
Suppose $M$ is a hyperbolic $3$--manifold homeomorphic to the interior of a compact manifold $\overline{M}$ with torus boundary components $T_1, \dots, T_k$. Suppose $s_1, \dots, s_k$ are slopes, with $s_i\subset T_i$. Suppose there exists a choice of horoball neighborhoods of the cusps of $M$ such that in the induced Euclidean metric on $T_i$, the slope $s_i$ has length strictly greater than $6$, for all $i$. Then the manifold obtained by Dehn filling along $s_1, \dots, s_k$, denoted $M(s_1, \dots, s_k)$, is hyperbolic. 
\end{theorem}

 \refthm{6Theorem} can be used to prove that a knot or link is hyperbolic, is as follows. First, show the knot complement $S^3\setminus K$ is obtained by Dehn filling a manifold $Y$ that is known to be hyperbolic. Then, prove that the slopes used to obtain $S^3\setminus K$ from $Y$ have length greater than $6$ on a horoball neighborhood of the cusps of $Y$.  See also \refsec{Cusps} for ways to prove that slopes are long.

\medskip

Some examples of links to which this theorem has been applied include \emph{highly twisted links}, which have diagrams with $6$ or more crossings in every twist region. (See \refdef{TwistNumber}.)
These links can be obtained by surgery, as follows. Start with a fully augmented link as described above, for instance the example shown in \reffig{Augmented}. Performing a Dehn filling along the slope $1/n$ on a crossing circle adds $2n$ crossings to the twist region encircled by that crossing circle, and removes the crossing circle from the diagram. 
When $|n| \geq 3$, the result of such Dehn filling on each crossing circle is highly twisted.

Using the explicit geometry of fully augmented links obtained from the circle packing, we may bound lengths of the slopes $1/n_i$ on crossing circles. Then \refthm{6Theorem} shows that the resulting knots and links must be hyperbolic \cite{FuterPurcell}.

Other examples can also be obtained in this manner. For example, Baker showed that infinite families of Berge knots are hyperbolic by showing they are Dehn fillings of minimally twisted chain link complements, which are known to be hyperbolic, along sequences of slopes that are known to grow in length  \cite{Baker:HypBerge}. 

The $6$--Theorem is sharp. This was shown by Agol \cite{Agol:Bounds}, and by Adams and his students for a knot complement \cite{AdamsEtAl:sharp}. The pretzel knot $P(n,n,n)$, which has $3$ twist regions, and the same number of crossings in each twist region, has a toroidal Dehn filling along a slope with length exactly $6$. 

\subsection{Fibered knots and high distance knots}

We finish this section with a few remarks about other ways to prove manifolds are hyperbolic, and give references for further information. However, these methods seem less directly applicable to knots in $S^3$ than those discussed above, and the full details are beyond the scope of this paper.

Recall \refdef{Fibered} of a fibered knot. When the monodromy is pseudo-Anosov, the knot complement is known to be hyperbolic \cite{Thurston:Fiber}. The figure--8 knot complement can be shown to be hyperbolic in this way; see for example~\cite[page~70]{thurston:notes}. Certain links obtained as the complement of closed braids and their braid axis have also been shown to be hyperbolic using these methods \cite{HironakaKin}. It seems difficult to apply these methods directly to knots, however.

Another method is to consider bridge surfaces of a knot. Briefly, there is a notion of distance that measures the complexity of the bridge splitting of a knot. Bachman and Schleimer proved that any knot whose bridge distance is at least 3 must be hyperbolic \cite{BachmanSchleimer}. It seems difficult to bound bridge distance for classes of examples directly from a knot diagram. Recent work of Johnson and Moriah is the first that we know to obtain such bounds \cite{JohnsonMoriah}.

\section{Volumes}\label{Sec:Volume}

As mentioned in the introduction, the goal of effective geometrization is 
 to determine or estimate geometric invariants directly from a diagram. As volume is the first and most natural invariant of a hyperbolic manifold, the problem of estimating volume from a diagram has received considerable attention. In this section, we survey some of the results and techniques on both upper and lower bounds on volume.

\subsection{Upper bounds on volume}
Many bounds in this section involve constants with geometric meaning. In particular, we define
\[
\vtet = \text{volume of a regular ideal tetrahedron in } \HH^3  = 1.0149 \ldots
\]
and
\[
\voct = \text{volume of a regular ideal octahedron in } \HH^3 = 3.6638 \ldots
\]
These constants are useful in combinatorial upper bounds on volume because every geodesic tetrahedron in $\HH^3$ has volume at most $\vtet$, and every geodesic octahedron has volume at most $\voct$. See e.g.\ Benedetti and Petronio \cite{BenedettiPetronio}.

\subsubsection{Bounds in terms of crossing number}
The first volume bounds for hyperbolic knots are due to Adams \cite{Adams:thesis}. 
He showed that, if $D = D(K)$ is a diagram of a hyperbolic knot or link with $c \geq 5$ crossings, then 
\begin{equation}\label{Eqn:AdamsFirst}
\vol(S^3 \setminus K) \leq 4 (c(D) - 4) \vtet.
\end{equation}
Adams' method of proof was to use the knot diagram to divide $S^3 \setminus K$ into tetrahedra with a mix of ideal and material vertices, and to count the tetrahedra. Since the subdivision contains at most $4(c(D) - 4)$ tetrahedra, and each tetrahedron has volume at most $\vtet$, the bound follows.

In a more recent paper \cite{Adams:triple-crossing-number}, Adams improved the upper bound of \refeqn{AdamsFirst}:

\begin{theorem}\label{Thm:AdamsV8}
Let $D = D(K)$ be a diagram of a hyperbolic link $K$, with at least $5$ crossings. Then
\[
\vol(S^3 \setminus K) \leq  (c(D) - 5) \voct + 4 \vtet.
\]
\end{theorem}

Again, the method is to divide the link complement into a mixture of tetrahedra and octahedra, and to bound the volume of each polyhedron by $\vtet$ or $\voct$ respectively. The subdivision into octahedra was originally described by D.\ Thurston.

\begin{figure}
\begin{overpic}{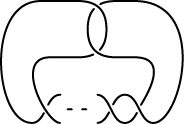}
\put(95,0){$K_n$}
\end{overpic}
\hspace{0.5in}
\begin{overpic}{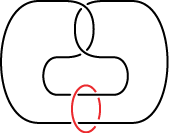}
\put(95,0){$L$}
\end{overpic}
\caption{Every twist knot $K_n$ has two twist regions, consisting of $2$ and $n$ crossings. Every $K_n$ can be obtained by Dehn filling the red component of the Whitehead link $L$, depicted on the right.}
\label{Fig:TwistKnot}
\end{figure}

The upper bound of \refthm{AdamsV8} is known to be asymptotically sharp, in the sense that there exist diagrams of knots and links $K_n$ with $\vol(S^3\setminus K_n)/c(K_n) \to \voct$ as $n\to \infty$; see  \cite{CKP:gmax}. On the other hand, this upper bound can be arbitrarily far from sharp. A useful example is the sequence of twist knots $K_n$ depicted in \reffig{TwistKnot}. Since the number of crossings is $n+2$, the upper bound of \refthm{AdamsV8} is linear in $n$. However, the volumes of $K_n$ are universally bounded and only increasing to an asymptotic limit:
\[
\vol(S^3 \setminus K_n) < \voct, \qquad \lim_{n \to \infty} \vol(S^3 \setminus K_n) = \voct
\]
This holds as a consequence of the following theorem of Gromov and Thurston \cite[Theorem 6.5.6]{thurston:notes}.

\begin{theorem}\label{Thm:VolumeDecreases}
Let $M$ be a finite volume hyperbolic manifold with cusps. Let $N = M(s_1, \ldots, s_n)$ be a Dehn filling of some cusps of $M$. Then $\vol(N) < \vol(M)$.
\end{theorem}

Returning to the case of twist knots, every $K_n$ can be obtained by Dehn filling on one component of the Whitehead link $L$, shown in \reffig{TwistKnot}, right. \refthm{VolumeDecreases} implies
$\vol(S^3 \setminus K_n) < \vol(S^3 \setminus L) = \voct$.

\subsubsection{Bounds in terms of twist number}

Following the example of twist knots in \reffig{TwistKnot}, it makes sense to seek upper bounds on volume in terms of the twist number $t(K)$ of a knot $K$ (see \refdef{TwistNumber}), rather than the crossing number alone. 

The following result combines the work of Lackenby \cite{lackenby:alt-volume} with an improvement by Agol and D.\ Thurston \cite[Appendix]{lackenby:alt-volume}.

\begin{theorem}\label{Thm:LATUpperBound}
Let $D(K)$ be a diagram of a hyperbolic link $K$. Then
\[
\vol(S^3 \setminus K) \leq 10 (t(D) - 1) \vtet.
\]
Furthermore, this bound is asymptotically sharp, in the sense that there exist knot diagrams $D_n = D(K_n)$ with $\vol(S^3 \setminus K_n)/t(D_n) \to 10 \vtet$.
\end{theorem}

The method of proof is as follows. First, one constructs a \emph{fully augmented link} $L$, by adding an extra component for each twist region of $D(K)$ (see \reffig{Augmented}). As described in \refsec{CirclePack}, the  link complement $S^3 \setminus L$ has simple and explicit combinatorics, making it relatively easy to bound $\vol(S^3 \setminus L)$ by counting tetrahedra. Then,
 \refthm{VolumeDecreases} implies that the same upper bound on volume applies to $S^3 \setminus K$.

As a counterpart to the asymptotic sharpness of \refthm{LATUpperBound}, there exist sequences of knots where $t(K_n) \to \infty$ but $\vol(S^3 \setminus K_n)$ is universally bounded. One family of such examples is the \emph{double coil knots} studied by the authors \cite{fkp:coils}.

Subsequent refinements or interpolations between \refthm{AdamsV8} and \refthm{LATUpperBound}  have been found by Dasbach and Tsvietkova \cite{DasbachTsvietkova, DasbachTsvietkova:simplicial} and Adams \cite{Adams:bipyramids}. These refinements produce a smaller upper bound compared to that of \refthm{LATUpperBound} when the diagram $D(K)$ has both twist regions with many crossings and with few crossings. However, the worst case scenario for the multiplicative constant does not improve due to the asymptotic sharpness of Theorems~\ref{Thm:AdamsV8} and~\ref{Thm:LATUpperBound}.

\subsection{Lower bounds on volume}

By results of Jorgensen and Thurston \cite{thurston:notes}, the volumes of hyperbolic $3$--manifolds are well-ordered. It follows that every family of hyperbolic $3$--manifolds (e.g.\ link complements; fibered knot complements, knot complements of genus 3, etc.) contains finitely many members realizing the lowest volume. Gabai, Meyerhoff, and Milley \cite{gmm:smallest-cusped} showed that the three knot complements of lowest volume are the figure-8 knot, the $5_2$ knot, and the $(-2,3,7)$ pretzel, whose volumes are
\begin{equation}\label{Eqn:LowestKnots}
\vol(4_1) = 2 \vtet = 2.0298\ldots, \qquad
\vol( 5_2) = \vol( P(-2,3,7)) = 2.8281 \ldots.
\end{equation}
Agol \cite{agol:multicusp} showed that the two multi-component links of lowest volume are the Whitehead link and the $(-2,3,8)$ pretzel link, both of which have volume $\voct = 3.6638 \ldots$. 
Yoshida \cite{Yoshida:Smallest4Cusped} has identified the smallest volume link of $4$ components, with volume $2 \voct$. 
Beyond these entries, lower bounds applicable to \emph{all} knots  (or  \emph{all} links) become scarce. Not even the lowest volume link of $3$ components is known to date.

Nevertheless, there are several practical methods of obtaining diagrammatic lower bounds on the volume of a knot or link, each of which applies to an infinite family of links, and each of which produces \emph{scalable} lower bounds that become larger as the complexity of a diagram becomes larger.
We survey these methods below.

\subsubsection{Angle structures} Suppose that $S^3 \setminus K$ has an ideal triangulation $\tau$ supporting an  angle structure $\theta$.
(Recall \refdef{AngleStruct}.) Every ideal tetrahedron of $\tau$, supplied with angles via $\theta$, has an associated volume. As a consequence, one may naturally define a volume $\vol(\theta)$ by summing the volumes of the individual tetrahedra.

\begin{conjecture}[Casson]\label{Conj:Casson}
Let $\tau$ be an ideal triangulation of a hyperbolic manifold $M$, which supports an angle structure $\theta$. Then
\[
\vol(\theta) \leq \vol(M),
\]
with equality if and only if $\theta$ solves the gluing equations and gives the complete hyperbolic structure.
\end{conjecture}

While \refconj{Casson} is open in general, it is known to hold if the triangulation $\tau$ is \emph{geometric}, meaning that some (possibly different) angle structure $\theta'$ solves the gluing equations on $\tau$. In this case, a theorem of Casson and Rivin \cite{fg:angled-survey, rivin:volume} says that $\theta'$ uniquely maximizes volume over all angle structures on $\tau$, implying in particular that $\vol(\theta) \leq \vol(\theta') = \vol(M)$.

In particular, the known case of \refconj{Casson} has been applied to the family of $2$--bridge links. In this case, the link complement has a natural angled triangulation whose combinatorics is closely governed by the link diagram \cite[Appendix]{GueritaudFuter:2bridge}. It follows that, for a sufficiently reduced diagram $D$ of a $2$--bridge link $K$,
\begin{equation}\label{Eqn:2BridgeVolume}
2 \vtet t(D) - 2.7066 \leq \vol(S^3 \setminus K) \leq 2\voct (t(D) - 1),
\end{equation}
which both sharpens the upper bound of \refthm{LATUpperBound} and proves a comparable lower bound.

There are rather few other families where this method has been successfully applied. One is the weaving knots studied by Champanerkar, Kofman, and Purcell \cite{CKP:Weaving}.

In the spirit of open problems, we mention the family of fibered knots and links. Agol showed that these link complements admit combinatorially natural \emph{veering triangulations} \cite{Agol:veering}, which have angle structures with nice properties \cite{HRST:veering, FG:veering}. A proof of \refconj{Casson}, even for this special family, would drastically expand the list of link complements for which we have practical, combinatorial volume estimates. See Worden \cite{Worden:experiment} for more on this problem.

\subsubsection{Guts} One powerful method of estimating the volume of a Haken $3$--manifold was developed by Agol, Storm, and Thurston \cite{AST:guts}, building on previous work of Agol \cite{Agol:guts}.

\begin{definition}
Let $M$ be a Haken hyperbolic $3$--manifold and $S \subset M$ a properly embedded essential surface. We use the symbol $M \cut S$ to denote the complement in $M$ of a collar neighborhood of $S$. Following the work of Jaco, Shalen, and Johannson \cite{jaco-shalen, johannson}, there is a canonical way to decompose $M \cut S$ along essential annuli into three types of pieces:
\begin{itemize}
\item $I$--bundles over a subsurface $\Sigma \subset S$,
\item Seifert fibered pieces, which are necessarily solid tori when $M$ is hyperbolic,
\item all remaining pieces, which are denoted $\guts(M,S)$.
\end{itemize}
Thurston's hyperbolization theorem (a variant of \refthm{Haken}) implies that $\guts(M,S)$ admits a hyperbolic metric with totally geodesic boundary. By Miyamoto's theorem \cite{miyamoto},  this metric with geodesic boundary has volume at least $\voct \abs{ \chi(\guts(M,S))  }$, where $\chi$ denotes Euler characteristic. 
\end{definition}

Agol, Storm, and Thurston showed \cite{AST:guts}:

\begin{theorem}\label{Thm:ASTGuts}
Let $M$ be a Haken hyperbolic $3$--manifold and $S \subset M$ a properly embedded essential surface. Then
\[
\vol(M) \geq \voct \abs{ \chi(\guts(M,S))  }.
\]
\end{theorem}

The proof of \refthm{ASTGuts} relies on geometric estimates due to Perelman. Agol, Storm, and Thurston double $M \cut S$ along its boundary and apply Ricci flow with surgery. 
They show that the metric on $\guts(M,S)$ converges to the one with totally geodesic boundary, while volume decreases, and while the metric on the remaining pieces shrinks away to volume $0$.

\refthm{ASTGuts} has been applied to several large families of knots. For alternating knots and links, Lackenby computed the guts of checkerboard surfaces in an alternating diagram \cite{lackenby:alt-volume}. Combined with
 \refthm{LATUpperBound} and \refthm{ASTGuts}, this implies: 

\begin{theorem}\label{Thm:LackenbyVolAlt}
  Let $D$ be a prime alternating diagram of a hyperbolic link $K$ in $S^3$. Then
\begin{equation*}
\frac{\voct}{2} (t(D)-2) \leq \vol(S^3 \setminus K) \leq 10\vtet (t(D) - 1),
\end{equation*}
\end{theorem}

Thus, for alternating knots, the combinatorics of a diagram determines $\vol(S^3 \setminus K)$ up to a factor less than $6$. Compare \refeqn{2BridgeVolume} in the 2--bridge case.

The authors of this survey have extended the method to the larger family of \emph{semi-adequate} links, and the even larger family of \emph{homogeneously adequate} links. Refer to \cite{fkp:guts} and \cite{fkp:PolandSurvey} for definitions of these families and precise theorem statements. The method gives particularly straightforward estimates in the same vein as \refthm{LackenbyVolAlt} for positive braids \cite{fkp:guts, Giambrone} and for Montesinos links \cite{fkp:guts, FinlinsonPurcell}.

\begin{question}\label{Ques:GutsComputable}
Does every knot $K \subset S^3$ admit an essential  spanning surface $S$ such that the Euler characteristic $\chi(\guts(S^3 \setminus K, \, S))$ can be computed directly from diagrammatic data?
\end{question}

The answer to \refques{GutsComputable} is ``yes'' whenever $K$ admits a \emph{homogeneously adequate} diagram in the terminology of \cite{fkp:guts}. However, it is not known whether $K$ always admits such a diagram. This is closely related to \cite[Question 10.10]{fkp:guts}.

\subsubsection{Dehn filling bounds}
A powerful method for proving lower bounds on the volume of $N = S^3 \setminus K$ involves two steps: first, prove a lower bound on $\vol(M)$ for some surgery parent $M$ of $N$, using one of the above methods; and second, control the change in volume as we Dehn fill $M$ to recover $N$.

The following theorem, proved in \cite{fkp:filling}, provides an estimate that has proved useful for lower bounds on the volume of knot complements.

\begin{theorem}\label{Thm:VolChange} 
Let  $M$ be a cusped hyperbolic $3$--manifold, containing embedded horocusps $C_1, \ldots, C_k$ (plus possibly others). On each torus $T_i = \bdy C_i$, choose a slope $s_i$, such that the shortest length of any of the $s_i$ is $\lmin > 2\pi$. 
  Then the manifold $M(s_1, \dots, s_k)$
obtained by Dehn filling along $s_1, \dots, s_k$ is hyperbolic, and its volume satisfies
\[
 \vol(M(s_1, \dots, s_k)) \geq 
  \left(1-\left(\frac{2\pi}{\lmin}\right)^2\right)^{3/2} \vol(M).
  \]
\end{theorem}

Earlier results in the same vein include an asymptotic estimate by Neumann and Zagier \cite{NeumannZagier}, as well as a cone-deformation estimate by Hodgson and Kerckhoff \cite{HK:univ}.

The idea of the proof of \refthm{VolChange} is as follows. Building on the proof of the Gromov--Thurston $2\pi$-Theorem, construct explicit negatively curved metrics on the solid tori added during Dehn filling. This yields a negatively curved metric on $M(s_1, \dots, s_k)$ whose volume is bounded below in terms of $\vol(M)$. Then, results of  
Besson, Courtois, and Gallot \cite{BCG, BCS} can be used to compare the volume of the negatively curved metric on $M(s_1, \dots, s_k)$ with the true hyperbolic volume.

Theorem \ref{Thm:VolChange} 
leads to diagrammatic volume bounds for several classes of hyperbolic links. 
For example, the following theorem from \cite{fkp:filling} gives a double-sided volume bound similar to \refthm{LackenbyVolAlt}.

\begin{theorem} \label{htwisted}
Let $K \subset S^3$ be a link with a prime, twist--reduced diagram
  $D(K)$. 
 
  Assume that $D(K)$ has $t(D) \geq 2$ twist regions, and
  that each region contains at least $7$ crossings. Then $K$ is a
  hyperbolic link satisfying
  $$0.70735 \; (t(D) - 1) \; < \; \vol(S^3 \setminus K) \; < \; 10\,
  \vtet \, (t(D) - 1).$$
  \end{theorem}
  
  The strategy of the proof of Theorem \ref{htwisted} is to view $S^3 \setminus K$
  as a Dehn filling on the complement of an augmented link obtained from the highly twisted diagram
  $D(K)$.  The   volume of augmented links can be bounded below in terms of $t(D)$ using Miyamoto's theorem \cite{miyamoto}.  The hypothesis that each region contains at least $7$ crossings ensures that the filling slopes are strictly longer than $2\pi$, hence \refthm{VolChange} gives the result.
  
Similar arguments using \refthm{VolChange} have been applied to links obtained by adding alternating tangles \cite{fkp:symmetric}, closed 3--braids \cite{fkp:farey} and weaving links \cite{CKP:Weaving}. 

\subsubsection{Knots with symmetry groups}

We close this section with a result about the volumes of symmetric knots. Suppose $K \subset S^3$ is a hyperbolic knot, and $G$ is a group of symmetries of $K$. That is, $G$ acts on $S^3$ by orientation--preserving homeomorphism that send $K$ to itself. It is a well-known consequence of Mostow rigidity that $G$ is finite and acts on $M = S^3 \setminus K$ by isometries \cite[Corollary 5.7.4]{thurston:notes}. Furthermore, $G$ is cyclic or dihedral \cite{HTW}. 

Define $n=n(G)$ to be the smallest order of a subgroup $\mathrm{Stab}_G(x)$ stabilizing a point $x \in S^3 \setminus K$, or else $n = |G|$ if the group acts freely. While this definition depends on how $G$ acts, it is always the case that $n(G)$ is at least as large as the smallest prime factor of $|G|$.

The following result follows by combining several statements in the literature. Since it has not previously been recorded, we include a proof.

\begin{theorem}\label{Thm:SymmetricKnotVolume}
  Let $K \subset S^3$ be a hyperbolic knot. Let $G$ be a group of orientation--preserving symmetries of $S^3$ that send $K$ to itself. Define $n=n(G)$ as above.
  Then
  \begin{equation*}
    \vol(S^3 \setminus K) \: \geq \: |G| \cdot x_n,
  \end{equation*}
  where $x_n = 2.848$ if $n>10$ and $n\neq 13, 18, 19$ and $x_n$ takes the following values otherwise.

\begin{center}
  \begin{tabular}{| l c | c |l l |}
    \hline 
    $\voct / 12 = 0.30532 \ldots$ & $n =2$ & \hspace{.2in} 
      & $2.16958$ & $n =7,8$ \\
    $\vtet / 2 = 0.50747 \ldots$ & $n =3$ & & 
	$2.47542$ & $n =9$ \\
    $0.69524$ & $n =4$ & &
        $2.76740$ & $n = 10$ \\
    $1.45034$ & $n =5$ & & 
	$\vol({\tt m011}) = 2.7818\dots$ & $n = 13$ \\
    $2.00606$ & $n =6$  & & 
	$\vol({\tt m016}) = 2.8281\dots$ & $n = 18,19$ \\
    \hline
  \end{tabular}
\end{center}
  \end{theorem}
	
	\begin{proof}

First, suppose that $G$ acts on $M = S^3 \setminus K$ with fixed points. Then the quotient $\mathcal{O} = M / G$ is a non-compact, orientable hyperbolic $3$--orbifold whose torsion orders are bounded below by $n$. We need to check that $\vol(\mathcal{O}) \geq x_n$. If $n = 2$, this result is due to Adams \cite[Corollary 8.2]{adams:limit-volume-orbifolds}. If $n = 3$, the result is essentially due to Adams and Meyerhoff; see \cite[Lemma 2.2]{atkinson-futer:link-orbifolds} and \cite[Lemma 2.3]{atkinson-futer:high-torsion}. If $n \geq 4$, the result is due to Atkinson and Futer  \cite[Theorem 3.8]{atkinson-futer:high-torsion}. In all cases, it follows that $\vol(M) \geq |G| \cdot x_n$.	
	
Next, suppose that $G$ acts freely on $M = S^3 \setminus K$. Then the quotient $N = M / G$ is a non-compact, orientable hyperbolic $3$--manifold. If $\vol(N) \geq 2.848$, then the theorem holds automatically because $x_n \leq 2.848$ for all $n$. If $\vol(N) < 2.848$, then Gabai, Meyerhoff, and Milley showed that $N$ is one of  10 enumerated $3$--manifolds \cite[Theorem 1.2]{gmm:smallest-cusped}. In SnapPy notation, these are $\tt m003$,
$\tt m004$, $\tt m006$, $\tt m007$, $\tt m009$, $\tt m010$, $\tt m011$, $\tt m015$, $\tt m016$, and $\tt m017$.
    We restrict attention to these manifolds.

Since $G$ acts freely on $M$, the solution to the Smith conjecture implies that $G$ also acts freely on $S^3$. 
By a theorem of Milnor \cite[page 624]{Milnor:free-action}, $G$ contains at most one element of order $2$, which implies that it must be cyclic. Thus $P = S^3/G$ is a lens space obtained by a Dehn filling on $N$. 
An enumeration of the lens space fillings of the 10 possible manifolds $N$ appears in the table on \cite[page 243]{fkp:symmetric}. This enumeration can be used to show that all possibilities satisfy the statement of the theorem.

Suppose that a lens space $L(p,q)$ is a Dehn filling of $N$. If $N$ actually occurs as a quotient of $M = S^3 \setminus K$, then $M$ must be a cyclic $p$--fold cover of $N$. We may rigorously enumerate all cyclic $p$--fold covers using  SnapPy \cite{SnapPy}. In almost all cases, homological reasons show that these covers are not knot complements. For instance, $N = \mathtt{m003}$ has two lens space fillings: $L(5,1)$ and $L(10,3)$. This manifold has six $5$--fold and six $10$--fold cyclic covers, none of which has first homology $\ZZ$. Thus  $\mathtt{m003}$ is not a quotient of a knot complement. The same technique applies to 8 of the 10 manifolds $N$. 

The two remaining exceptions determine several values of $x_n$.
The manifold $ {\tt m011}$ has $9$--fold and $13$--fold cyclic covers that are knot complements in $S^3$. The value of $x_9$ is already smaller than $\vol({\tt m011})$, but the value of $x_{13}$ is determined by this example. Similarly, the manifold ${\tt m016}$, which is the $(-2,3,7)$ pretzel knot complement, has $18$--fold and $19$--fold cyclic covers that are knot complements, determining the values of $x_{18}$ and $x_{19}$.
\end{proof}

\section{Cusp shapes and cusp areas}\label{Sec:Cusps}

Several results discussed above, such as  \refthm{6Theorem} and  \refthm{VolChange}, require the slopes used in Dehn filling along knot or link complements to be long. To bound lengths of slopes, we consider an additional invariant of hyperbolic knots and links, namely their cusp shapes and cusp areas. 

\begin{definition}\label{Def:CuspArea}
Let $C_1, \dots, C_n$ be a fixed choice of maximal cusps for a link complement $M$, as in \refdef{MaximalCusp}. The \emph{cusp area} of a component $C_i$, denoted by $\area(\partial C_i)$ is the Euclidean area of $\partial C_i$. The \emph{cusp volume}, denoted by $\vol(C_i)$, is the volume of $C_i$. Note that $\area(\partial C_i) = 2 \vol(C_i)$. When $M$ has multiple cusps, the cusp area and cusp volume depend on the choice of maximal cusp.
\end{definition}

This section surveys some methods for estimating the area of a maximal cusp and the length of slopes on it, and poses some open questions.

\subsection{Direct computation}

Similar to the techniques in \refsec{Hyperbolic}, if we can  explicitly  determine a geometric triangulation of a hyperbolic $3$--manifold, then we can determine its cusp shape and cusp area. This is implemented in SnapPy \cite{SnapPy}.

For fully augmented links, whose geometry is completely determined by a circle packing, the cusp shape is also determined by the circle packing. The cusp area can be computed by finding an explicit collection of disjoint horoballs in the fully augmented link, as in \cite{FuterPurcell}.

\medskip

Under very strong hypotheses, it is possible to apply the cone deformation techniques of Hodgson and Kerckhoff \cite{HK:univ} to bound the change in cusp shape under Dehn filling. 
Purcell carried this out in \cite{Purcell:Cusps}, starting from a fully augmented link. However, the results only apply to knots with at least two twist regions and at least 116 crossings per twist region.

To obtain more general bounds for larger classes of knots and links, additional tools are needed. The main tools are pleated surfaces and packing techniques. 

\subsection{Upper bounds and pleated surfaces}

If $M$ is a hyperbolic link complement, then for any choice of maximal cusp, there is a collection of slopes whose Dehn fillings gives $S^3$. These are the \emph{meridians} of $M$. Because $S^3$ is not hyperbolic, the $6$--Theorem implies that in any choice of maximal cusp for $M$,  one or more of these slopes must have length at most $6$. 
Indeed, the $6$--Theorem is proved by considering punctured surfaces  immersed in $M$ and using area arguments to bound the length of a slope. 

\begin{definition}\label{Def:Pleated}
Let $M$ be a hyperbolic $3$--manifold with cusps a collection of cusps $C$, and let $S$ be a hyperbolic surface. A \emph{pleated surface} is a piecewise geodesic, proper  immersion $f\from S \to M$. Properness means that any cusps of $S$ are mapped into cusps of $M$. 
The surface $S$ is cut into ideal triangles, each of which is mapped isometrically into $M$. In $M$, there may be bending along the sides of the triangles. See \reffig{Pleating}.
\end{definition}

\begin{figure}
\includegraphics{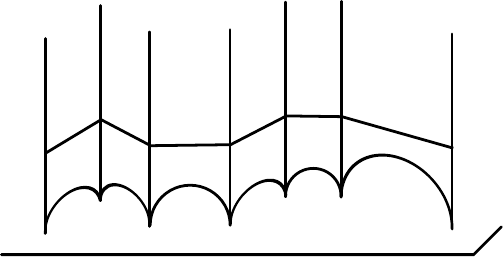}
\caption{The lift of a pleated surface to the universal cover $\HH^3$ of $M$. The piecewise linear zig-zag lies in a single horosphere.}
\label{Fig:Pleating}
\end{figure}

An essential surface $S$ in a hyperbolic $3$--manifold $M$ can always be homotoped into a pleated form.
The idea is to start with an ideal triangulation of $S$, then homotope the images of the edges in $M$ to be ideal geodesics in $M$. Similarly, homotope the ideal triangles to be totally geodesic, with sides the geodesic edges in $M$. This gives $S$ a pleating. See 
 \cite[Theorem 5.3.6]{CanaryEpsteinGreen:Notes} or  \cite[Lemma 2.2]{Lackenby:Word} for proofs. 

The main result on slope lengths and pleated surfaces is the following, which is a special case of \cite[Theorem~5.1]{Agol:Bounds} and  \cite[Lemma~3.3]{Lackenby:Word}.
The result  is used in the proof of the $6$--Theorem.

\begin{theorem}\label{Thm:Pleated}
Let $M=S^3\setminus K$ be a hyperbolic knot complement with a maximal cusp $C$. Suppose that $f\from S \to M$ is a pleated surface, and let $\ell_C(S)$ denote the total length of the intersection curves in $f(S) \cap \partial C$. Then 
\[ \ell_C(S)\ \leq \  6 | \chi(S)|.\]
\end{theorem}

The idea of the proof of \refthm{Pleated} is to find disjoint horocusp  neighborhoods $H=\cup H_i$ in $S$ such that $f(H_i)\subset C$,  and such that $\ell(\partial H_i)$ is at least as big as the length of $f(\partial H_i)$ measured on $C$. This allows us to compute as follows:
\[ \ell_C(S)\leq \sum_{i=1}^{s}\ell(\partial H_i) = \sum_{i=1}^{s} \area(H_i) \leq  \frac{6}{2\pi} \area(S)=\frac{6}{2\pi} \cdot 2\pi  | \chi(S)|. \]
Here, the first inequality is by construction. The second equality is a general fact about hyperbolic surfaces, proved by a calculation in $\HH^2$. The third inequality is a packing theorem due to B\"or\"oczky \cite{boroczky}. The final equality is the Gauss--Bonnet theorem.
\subsubsection{Sample applications}

As noted above, the 6--Theorem implies that the length of a meridian is at most $6$.
\refthm{Pleated} has also been used to estimate the lengths of other slopes. For example, a \emph{$\lambda$--curve} is defined to be a curve that intersects the meridian $\mu$ exactly once. The knot-theoretic longitude, which is null-homologous in $S^3 \setminus K$, is one example of a $\lambda$--curve, and need not be the shortest $\lambda$--curve.
There may be one or two shortest $\lambda$--curves. For any $\lambda$--curve $\lambda$, note that $\ell(\mu)\ell(\lambda)$ gives an upper bound on cusp area. 

By applying \refthm{Pleated} to a singular spanning surface in a knot complement, the authors of \cite{AdamsEtAl:CuspSize} obtain the following upper bounds on meridian, $\lambda$--curve, and cusp area.

\begin{theorem}\label{Thm:AdamsCuspBound}
Let $K$ be a hyperbolic knot in $S^3$ with crossing number $c=c(K)$. Let $C$ denote the maximal cusp of $S^3\setminus K$. Then, for the meridian $\mu$ and for the shortest $\lambda$--curve,
\[ \ell(\mu) \leq 6 - \frac{7}{c}, \quad \ell(\lambda) \leq 5c-6, \quad
\mbox{ and } \quad \area(\partial C) \leq {9c} \left(1 - \dfrac{1}{c}\right)^2.
\]
\end{theorem}

Another instance where \refthm{Pleated} applies is to knots with a pair of essential spanning surfaces $S_1$ and $S_2$; in this case the surface $S$ is taken to be the disjoint union of the two spanning surfaces. The following appears in \cite{BurtonKalf}. 

\begin{theorem}\label{Thm:GeneralMeridianBound}
Let $K$ be a hyperbolic knot with maximal cusp $C$. Suppose that $S_1$ and $S_2$ are essential spanning surfaces in $M=S^3 \setminus K$ and let $i(\partial S_1, \partial S_2)\neq 0$ denote the minimal intersection number of $\partial S_1, \partial S_2$ in $\partial C$. Finally, let $\chi = |\chi(S_1)|+|\chi(S_2)|$. 
Then, for the meridian $\mu$ and the shortest $\lambda$--curve,
\[ \ell(\mu) \leq \frac{6\chi}{i(\bdy S_1, \bdy S_2)}, \quad
\ell(\lambda) \leq 3 \chi, \quad \mbox{ and } \quad
\area(\bdy C) \leq \frac{18 \chi^2}{i(\bdy S_1, \bdy S_2).}
\]
\end{theorem}

\refthm{GeneralMeridianBound} is useful because the checkerboard surfaces of many knot diagrams are known to be essential.
For instance, the checkerboard surfaces of alternating diagrams are essential. Indeed, in \cite{AdamsEtAl:CuspSize} the authors use pleated checkerboard surfaces to prove the meridian of an alternating knot satisfies $\ell(\mu)<3-6/c$. Other knots with essential spanning surfaces include \emph{adequate knots}, which arose in the study of Jones type invariants. Ozawa first proved that two surfaces in such links are essential \cite{ozawa:adequate}; see also \cite{fkp:guts}. More generally, \refthm{GeneralMeridianBound}  applies to knots that admit alternating projections on surfaces so that they define essential checkerboard surfaces. These have been studied by Ozawa \cite{Ozawa} and Howie \cite{Howie:Thesis}.

All the results above indicate that meridian lengths should be strictly less than $6$. 
For knots in $S^3$, no examples are known with length more than $4$. 

\begin{question}\label{Ques:LengthMeridian}
Do all hyperbolic knots in $S^3$ satisfy $\ell(\mu)\leq 4$?
\end{question}

For links in $S^3$, Goerner showed there exists a link in $S^3$ with 64 components, and a choice of cusps for which each meridian length is $\sqrt{21}\approx 4.5826$ \cite{goerner}. 

\begin{question}\label{Ques:LengthLinkMeridians}
Given a hyperbolic link $L \subset S^3$, consider the shortest meridian among the components of $L$. What is the largest possible value of the shortest meridian? 
Is it $\sqrt{21}$?

\end{question}

The $6$--Theorem gives a bound on the length of any non-hyperbolic Dehn fillings. By geometrization, non-hyperbolic manifolds are either \emph{reducible} (meaning they contain an essential $2$--sphere), or \emph{toroidal} (meaning the contain an essential torus), or small Seifert fibered. The $6$--Theorem is only known to be sharp on toroidal fillings. Thus one may ask about the maximal possible length for the other types of fillings.
See \cite{HoffmanPurcell} for related questions and results.

\subsubsection{Upper bounds on area via cusp density}
The \emph{cusp density} of a cusped $3$--manifold $M$ is the volume of a maximal cusp divided by the volume of $M$. B\"or\"oczky  \cite{boroczky} showed that cusp density is universally bounded by $\sqrt{3}/2\vtet$, with the figure--8 knot complement realizing this bound. Recall from \refthm{LATUpperBound} that every hyperbolic knot $K \subset S^3$ satisfies $\vol(S^3\setminus K)\leq 10\, \vtet(t-1)$, where $t = t(D)$ is the twist number of any diagram. Combining this with B\"or\"oczky's theorem shows that a maximal cusp $C \subset S^3 \setminus K$ satisfies
\[ \area(\partial C)\leq 10 \sqrt 3 \cdot (t-1) \approx 17.32\cdot (t-1) .\]

We note that this bound can be arbitrarily far from sharp. This is already true for  \refthm{LATUpperBound}. In addition, Eudave-Mu{\~n}oz and Luecke  \cite{eudave-munoz-luecke} showed that the cusp density of a hyperbolic knot complement can be arbitrarily close to $0$.

\subsection{Lower bounds via horoball packing}

Theorems~\ref{Thm:Pleated} and~\ref{Thm:AdamsCuspBound} give methods for bounding cusp area from above. To give lower bounds on slope lengths, for example to apply the 6--Theorem, we must bound cusp area or cusp volume from below. The main tool for this is to use \emph{packing arguments}: find a disjoint collection of horoballs with Euclidean diameters bounded from below in a fundamental region of the cusp. Take their shadows on the cusp torus. The area of the cusp torus must be bounded below by the areas of the shadows.
One sample result is the following, from \cite{LP:AltCusps}.

\begin{lemma}\label{Lem:LPShortarcsVolume}
Suppose that a one-cusped hyperbolic $3$--manifold $M$ contains at least $p$ homotopically distinct essential arcs, each with length at most $L$ measured with respect to the maximal cusp $H$ of $M$. Then the cusp area $\area(\bdy H)$ is at least $p\,\sqrt{3}\,e^{-2L}$.
\end{lemma}

Similar techniques were also used to bound cusp areas in \cite{fkp:farey} and in \cite{futer-schleimer}.

The idea of the proof is that an arc from the cusp to itself of length $L$ lifts to an arc in the universal cover between two horoballs. We may identify the universal cover of $M$ with 
the upper half-space model of $\HH^3$, so that the boundary of one cusp  in $M$ lifts to a horosphere at Euclidean height $1$. The Euclidean metric on this horosphere coincides with the hyperbolic metric. Arcs of bounded length lead to horoballs whose diameter is not too small, and whose shadows have a definite area.

At this writing there is no general lower bound of cusp shapes for all hyperbolic knots. However, for alternating knots, Lackenby and Purcell found a collection of homotopically distinct essential arcs of bounded length, then applied \reflem{LPShortarcsVolume} to to show the following \cite{LP:AltCusps}.

\begin{theorem} \label{Thm:AltCusps}
Let $D$ be a prime, twist reduced alternating diagram of some hyperbolic knot $K$ and let $t = t(D)$ be the twist number of $D$. Let $C$ be the maximal cusp of $M=S^3\setminus K$. Then
\[ A (t-2) \leq \area(\partial C)  \leq 10 \sqrt{3} (t-1), \]
where $A$ is at least  $2.278 \times 10^{-9}$.
\end{theorem}

For 2--bridge knots there is a much sharper lower bound \cite{fkp:farey}: 
\[ \frac{8\sqrt{3}}{147} (t-1) \leq \area(\partial C)  \leq  \frac{ \sqrt{3} \voct}{\vtet}(t-1). \]

Note that \refthm{AltCusps}, along with equation \refthm{LackenbyVolAlt} implies that the cusp density of alternating knots is universally bounded below. This is not true for non-alternating knots~\cite{eudave-munoz-luecke}. It would be interesting to study the extent to which \refthm{AltCusps} can be generalized.

In general, we would like to know how to obtain many homotopically distinct arcs that can be used in \reflem{LPShortarcsVolume}. The arcs used in the proof of \refthm{AltCusps} lie on complicated immersed essential surfaces, described in \cite{LP:Twisted}. It is conjectured that much simpler \emph{crossing arcs} should play this role. 

\begin{definition}\label{Def:CrossingArc}
Let $K$ be a knot with diagram $D(K)$. A \emph{crossing arc} is an embedded arc $\alpha$ in $S^3$ with  $\partial \alpha  \subset K$, such that in $D(K)$, $\alpha$ projects to an unknotted embedded arc running from an overstrand to an understrand in a crossing.
\end{definition}

The following conjecture is due to Sakuma and Weeks \cite{SakumaWeeks}. 

\begin{conjecture}\label{Conj:Arcs}
In a reduced alternating diagram of a hyperbolic alternating link, every crossing arc is isotopic to a geodesic.
\end{conjecture}

\refconj{Arcs} is known for 2--bridge knots \cite[Appendix]{GueritaudFuter:2bridge} and certain closed alternating braids~\cite{Tsvietkova}. 

Computer experiments performed by Thistlethwaite and Tsvietkova \cite{thistle-tsviet} also support the following conjecture, which would give more information on the lengths of crossing arcs, hence more information on cusp areas. 

\begin{conjecture}\label{Conj:CrossingArcLength}
Crossing arcs in alternating knots have  length universally bounded above by $\log 8$.
\end{conjecture}

\bibliographystyle{amsplain}
\bibliography{references}

\providecommand{\bysame}{\leavevmode\hbox to3em{\hrulefill}\thinspace}
\providecommand{\MR}{\relax\ifhmode\unskip\space\fi MR }
\providecommand{\MRhref}[2]{%
  \href{http://www.ams.org/mathscinet-getitem?mr=#1}{#2}
}
\providecommand{\href}[2]{#2}
\begin{thebibliography}{10}

\bibitem{Adams:thesis}
Colin Adams, \emph{Hyperbolic structures on link complements}, Ph.D. thesis,
  University of Wisconsin, Madison, 1983.

\bibitem{Adams:survey}
\bysame, \emph{Hyperbolic knots}, Handbook of knot theory, Elsevier B. V.,
  Amsterdam, 2005, pp.~1--18. \MR{2179259}

\bibitem{Adams:triple-crossing-number}
\bysame, \emph{Triple crossing number of knots and links}, J. Knot Theory
  Ramifications \textbf{22} (2013), no.~2, 1350006. \MR{3037297}

\bibitem{Adams:bipyramids}
\bysame, \emph{Bipyramids and bounds on volumes of hyperbolic links}, Topology
  Appl. \textbf{222} (2017), 100--114. \MR{3630197}

\bibitem{AdamsEtAl:sharp}
Colin Adams, Hanna Bennett, Christopher Davis, Michael Jennings, Jennifer
  Kloke, Nicholas Perry, and Eric Schoenfeld, \emph{Totally geodesic {S}eifert
  surfaces in hyperbolic knot and link complements. {II}}, J. Differential
  Geom. \textbf{79} (2008), no.~1, 1--23. \MR{2414747 (2009b:57032)}

\bibitem{AdamsEtAl:CuspSize}
Colin Adams, A.~Colestock, J.~Fowler, W.~Gillam, and E.~Katerman, \emph{Cusp
  size bounds from singular surfaces in hyperbolic 3-manifolds}, Trans. Amer.
  Math. Soc. \textbf{358} (2006), no.~2, 727--741 (electronic). \MR{MR2177038
  (2006k:57041)}

\bibitem{adams:limit-volume-orbifolds}
Colin~C. Adams, \emph{Limit volumes of hyperbolic three-orbifolds}, J.
  Differential Geom. \textbf{34} (1991), no.~1, 115--141. \MR{1114455
  (92d:57029)}

\bibitem{Adams:ToroidallyAlt}
\bysame, \emph{Toroidally alternating knots and links}, Topology \textbf{33}
  (1994), no.~2, 353--369.

\bibitem{Adams:AlmostAlt}
Colin~C. Adams, Jeffrey~F. Brock, John Bugbee, Timothy~D. Comar, Keith~A.
  Faigin, Amy~M. Huston, Anne~M. Joseph, and David Pesikoff, \emph{Almost
  alternating links}, Topology Appl. \textbf{46} (1992), no.~2, 151--165.

\bibitem{Agol:guts}
Ian Agol, \emph{Lower bounds on volumes of hyperbolic {H}aken 3-manifolds},
  arXiv:math/9906182, 1999.

\bibitem{Agol:Bounds}
\bysame, \emph{Bounds on exceptional {D}ehn filling}, Geom. Topol. \textbf{4}
  (2000), 431--449. \MR{1799796 (2001j:57019)}

\bibitem{agol:multicusp}
\bysame, \emph{The minimal volume orientable hyperbolic 2-cusped 3-manifolds},
  Proc. Amer. Math. Soc. \textbf{138} (2010), no.~10, 3723--3732. \MR{2661571}

\bibitem{Agol:veering}
\bysame, \emph{Ideal triangulations of pseudo-{A}nosov mapping tori}, Topology
  and geometry in dimension three, Contemp. Math., vol. 560, Amer. Math. Soc.,
  Providence, RI, 2011, pp.~1--17. \MR{2866919}

\bibitem{AST:guts}
Ian Agol, Peter~A. Storm, and William~P. Thurston, \emph{{Lower bounds on
  volumes of hyperbolic Haken 3-manifolds}}, J. Amer. Math. Soc. \textbf{20}
  (2007), no.~4, 1053--1077, with an appendix by Nathan Dunfield.

\bibitem{atkinson-futer:link-orbifolds}
Christopher~K. Atkinson and David Futer, \emph{Small volume link orbifolds},
  Math. Res. Lett. \textbf{20} (2013), no.~6, 995--1016. \MR{3228616}

\bibitem{atkinson-futer:high-torsion}
\bysame, \emph{The lowest volume 3-orbifolds with high torsion}, Trans. Amer.
  Math. Soc. \textbf{369} (2017), no.~8, 5809--5827. \MR{3646779}

\bibitem{BachmanSchleimer}
David Bachman and Saul Schleimer, \emph{Distance and bridge position}, Pacific
  J. Math. \textbf{219} (2005), no.~2, 221--235. \MR{2175113 (2007a:57028)}

\bibitem{Baker:HypBerge}
Kenneth~L. Baker, \emph{Surgery descriptions and volumes of {B}erge knots.
  {II}. {D}escriptions on the minimally twisted five chain link}, J. Knot
  Theory Ramifications \textbf{17} (2008), no.~9, 1099--1120. \MR{MR2457838}

\bibitem{BenedettiPetronio}
Riccardo Benedetti and Carlo Petronio, \emph{Lectures on hyperbolic geometry},
  Universitext, Springer-Verlag, Berlin, 1992. \MR{MR1219310 (94e:57015)}

\bibitem{BCG}
G\'erard Besson, Gilles Courtois, and Sylvestre Gallot, \emph{Entropies et
  rigidit\'es des espaces localement sym\'etriques de courbure strictement
  n\'egative}, Geom. Funct. Anal. \textbf{5} (1995), no.~5, 731--799.
  \MR{1354289}

\bibitem{BleilerHodgson:2pi}
Steven~A. Bleiler and Craig~D. Hodgson, \emph{Spherical space forms and {D}ehn
  filling}, Topology \textbf{35} (1996), no.~3, 809--833.

\bibitem{BCS}
Jeffrey Boland, Chris Connell, and Juan Souto, \emph{Volume rigidity for finite
  volume manifolds}, Amer. J. Math. \textbf{127} (2005), no.~3, 535--550.
  \MR{2141643}

\bibitem{boroczky}
K\'aroly B{\"o}r{\"o}czky, \emph{Packing of spheres in spaces of constant
  curvature}, Acta Math. Acad. Sci. Hungar. \textbf{32} (1978), no.~3-4,
  243--261. \MR{MR512399 (80h:52014)}

\bibitem{BurtonKalf}
Stephan~D. Burton and Efstratia Kalfagianni, \emph{Geometric estimates from
  spanning surfaces}, Bull. Lond. Math. Soc. \textbf{49} (2017), no.~4,
  694--708.

\bibitem{CanaryEpsteinGreen:Notes}
Richard~D. Canary, David B.~A. Epstein, and Paul~L. Green, \emph{Notes on notes
  of {T}hurston}, Fundamentals of hyperbolic geometry: selected expositions,
  London Math. Soc. Lecture Note Ser., vol. 328, Cambridge Univ. Press,
  Cambridge, 2006, With a new foreword by Canary, pp.~1--115. \MR{2235710}

\bibitem{CKP:gmax}
Abhijit Champanerkar, Ilya Kofman, and Jessica~S. Purcell, \emph{Geometrically
  and diagrammatically maximal knots}, J. Lond. Math. Soc. (2) \textbf{94}
  (2016), no.~3, 883--908. \MR{3614933}

\bibitem{CKP:Weaving}
\bysame, \emph{Volume bounds for weaving knots}, Algebr. Geom. Topol.
  \textbf{16} (2016), no.~6, 3301--3323. \MR{3584259}

\bibitem{Choi}
Young-Eun Choi, \emph{Positively oriented ideal triangulations on hyperbolic
  three-manifolds}, Topology \textbf{43} (2004), no.~6, 1345--1371.
  \MR{2081429}

\bibitem{SnapPy}
Marc Culler, Nathan~M. Dunfield, Matthias Goerner, and Jeffrey~R. Weeks,
  \emph{Snap{P}y, a computer program for studying the geometry and topology of
  $3$-manifolds}, Available at \url{http://snappy.computop.org} (2017).

\bibitem{DasbachTsvietkova}
Oliver Dasbach and Anastasiia Tsvietkova, \emph{A refined upper bound for the
  hyperbolic volume of alternating links and the colored {J}ones polynomial},
  Math. Res. Lett. \textbf{22} (2015), no.~4, 1047--1060. \MR{3391876}

\bibitem{DasbachTsvietkova:simplicial}
\bysame, \emph{Simplicial volume of links from link diagrams},
  arXiv:1512.08316, 2015, Math. Proc. Camb. Phil. Soc., to appear.

\bibitem{eudave-munoz-luecke}
Mario Eudave-Mu{\~n}oz and John Luecke, \emph{Knots with bounded cusp volume
  yet large tunnel number}, J. Knot Theory Ramifications \textbf{8} (1999),
  no.~4, 437--446. \MR{1697382 (2000g:57007)}

\bibitem{FinlinsonPurcell}
Kathleen Finlinson and Jessica~S. Purcell, \emph{Volumes of {M}ontesinos
  links}, Pacific J. Math. \textbf{282} (2016), no.~1, 63--105. \MR{3463425}

\bibitem{FG:Arborescent}
David Futer and Fran{\c{c}}ois Gu{\'e}ritaud, \emph{Angled decompositions of
  arborescent link complements}, Proc. Lond. Math. Soc. (3) \textbf{98} (2009),
  no.~2, 325--364. \MR{2481951}

\bibitem{fg:angled-survey}
\bysame, \emph{From angled triangulations to hyperbolic structures},
  Interactions between hyperbolic geometry, quantum topology and number theory,
  Contemp. Math., vol. 541, Amer. Math. Soc., Providence, RI, 2011,
  pp.~159--182. \MR{2796632}

\bibitem{FG:veering}
\bysame, \emph{Explicit angle structures for veering triangulations}, Algebr.
  Geom. Topol. \textbf{13} (2013), no.~1, 205--235. \MR{3031641}

\bibitem{fkp:filling}
David Futer, Efstratia Kalfagianni, and Jessica~S. Purcell, \emph{Dehn filling,
  volume, and the {J}ones polynomial}, J. Differential Geom. \textbf{78}
  (2008), no.~3, 429--464. \MR{2396249}

\bibitem{fkp:symmetric}
\bysame, \emph{Symmetric links and {C}onway sums: volume and {J}ones
  polynomial}, Math. Res. Lett. \textbf{16} (2009), no.~2, 233--253.
  \MR{2496741}

\bibitem{fkp:farey}
\bysame, \emph{{Cusp areas of {F}arey manifolds and applications to knot
  theory}}, Int. Math. Res. Not. IMRN \textbf{2010} (2010), no.~23, 4434--4497.

\bibitem{fkp:coils}
\bysame, \emph{On diagrammatic bounds of knot volumes and spectral invariants},
  Geom. Dedicata \textbf{147} (2010), 115--130. \MR{2660569}

\bibitem{fkp:guts}
\bysame, \emph{Guts of surfaces and the colored {J}ones polynomial}, Lecture
  Notes in Mathematics, vol. 2069, Springer, Heidelberg, 2013.

\bibitem{fkp:PolandSurvey}
\bysame, \emph{Jones polynomials, volume and essential knot surfaces: a
  survey}, Knots in {P}oland. {III}. {P}art 1, Banach Center Publ., vol. 100,
  Polish Acad. Sci. Inst. Math., Warsaw, 2014, pp.~51--77. \MR{3220475}

\bibitem{fkp:hyp}
\bysame, \emph{Hyperbolic semi-adequate links}, Comm. Anal. Geom. \textbf{23}
  (2015), no.~5, 993--1030. \MR{3458811}

\bibitem{FuterPurcell}
David Futer and Jessica~S. Purcell, \emph{Links with no exceptional surgeries},
  Comment. Math. Helv. \textbf{82} (2007), no.~3, 629--664. \MR{2314056}

\bibitem{futer-schleimer}
David Futer and Saul Schleimer, \emph{Cusp geometry of fibered 3-manifolds},
  Amer. J. Math. \textbf{136} (2014), no.~2, 309--356. \MR{3188063}

\bibitem{gmm:smallest-cusped}
David Gabai, Robert Meyerhoff, and Peter Milley, \emph{Minimum volume cusped
  hyperbolic three-manifolds}, J. Amer. Math. Soc. \textbf{22} (2009), no.~4,
  1157--1215. \MR{2525782}

\bibitem{Giambrone}
Adam Giambrone, \emph{Combinatorics of link diagrams and volume}, J. Knot
  Theory Ramifications \textbf{24} (2015), no.~1, 1550001, 21. \MR{3319678}

\bibitem{goerner}
Matthias Goerner, \emph{Regular tessellation links}, arXiv:1406.2827, 2014.

\bibitem{GueritaudFuter:2bridge}
Fran{\c{c}}ois Gu{\'e}ritaud, \emph{On canonical triangulations of
  once-punctured torus bundles and two-bridge link complements}, Geom. Topol.
  \textbf{10} (2006), 1239--1284, With an appendix by David Futer. \MR{2255497}

\bibitem{HironakaKin}
Eriko Hironaka and Eiko Kin, \emph{A family of pseudo-{A}nosov braids with
  small dilatation}, Algebr. Geom. Topol. \textbf{6} (2006), 699--738.
  \MR{2240913}

\bibitem{HK:univ}
Craig~D. Hodgson and Steven~P. Kerckhoff, \emph{Universal bounds for hyperbolic
  {D}ehn surgery}, Ann. of Math. (2) \textbf{162} (2005), no.~1, 367--421.
  \MR{MR2178964}

\bibitem{HRS:AnglesHomology}
Craig~D. Hodgson, J.~Hyam Rubinstein, and Henry Segerman, \emph{Triangulations
  of hyperbolic 3-manifolds admitting strict angle structures}, J. Topol.
  \textbf{5} (2012), no.~4, 887--908. \MR{3001314}

\bibitem{HRST:veering}
Craig~D. Hodgson, J.~Hyam Rubinstein, Henry Segerman, and Stephan Tillmann,
  \emph{Veering triangulations admit strict angle structures}, Geom. Topol.
  \textbf{15} (2011), no.~4, 2073--2089. \MR{2860987}

\bibitem{HoffmanPurcell}
Neil~R. Hoffman and Jessica~S. Purcell, \emph{Geometry of planar surfaces and
  exceptional fillings}, Bull. Lond. Math. Soc. \textbf{49} (2017), no.~2,
  185--201. \MR{3656288}

\bibitem{HTW}
Jim Hoste, Morwen Thistlethwaite, and Jeff Weeks, \emph{The first 1,701,936
  knots}, Math. Intelligencer \textbf{20} (1998), no.~4, 33--48. \MR{MR1646740
  (99i:57015)}

\bibitem{Howie:Thesis}
Joshua~A. Howie, \emph{Surface-alternating knots and links}, Ph.D. thesis,
  University of Melbourne, 2015.

\bibitem{jaco-shalen}
William~H. Jaco and Peter~B. Shalen, \emph{Seifert fibered spaces in
  {$3$}-manifolds}, Mem. Amer. Math. Soc. \textbf{21} (1979), no.~220,
  viii+192. \MR{MR539411 (81c:57010)}

\bibitem{johannson}
Klaus Johannson, \emph{Homotopy equivalences of {$3$}-manifolds with
  boundaries}, Lecture Notes in Mathematics, vol. 761, Springer, Berlin, 1979.
  \MR{MR551744 (82c:57005)}

\bibitem{Johnson:WYSIWYG}
Jesse Johnson, \emph{{WYSIWYG Hyperbolic knots}}, Low Dimensional Topology
  Blog, \url{https://ldtopology.wordpress.com/2007/11/18/temporary/}.

\bibitem{JohnsonMoriah}
Jesse Johnson and Yoav Moriah, \emph{Bridge distance and plat projections},
  Algebr. Geom. Topol. \textbf{16} (2016), no.~6, 3361--3384. \MR{3584261}

\bibitem{Lackenby:Word}
Marc Lackenby, \emph{Word hyperbolic {D}ehn surgery}, Invent. Math.
  \textbf{140} (2000), no.~2, 243--282.

\bibitem{lackenby:alt-volume}
\bysame, \emph{The volume of hyperbolic alternating link complements}, Proc.
  London Math. Soc. (3) \textbf{88} (2004), no.~1, 204--224, With an appendix
  by Ian Agol and Dylan Thurston.

\bibitem{LP:AltCusps}
Marc Lackenby and Jessica~S. Purcell, \emph{Cusp volumes of alternating knots},
  Geom. Topol. \textbf{20} (2016), no.~4, 2053--2078. \MR{3548463}

\bibitem{LP:Twisted}
\bysame, \emph{Essential twisted surfaces in alternating link complements},
  Algebr. Geom. Topol. \textbf{16} (2016), no.~6, 3209--3270. \MR{3584257}

\bibitem{luo-tillmann:generalized-angles}
Feng Luo and Stephan Tillmann, \emph{Angle structures and normal surfaces},
  Trans. Amer. Math. Soc. \textbf{360} (2008), no.~6, 2849--2866. \MR{2379778}

\bibitem{Menasco:Polyhedra}
William~W. Menasco, \emph{Polyhedra representation of link complements},
  Low-dimensional topology ({S}an {F}rancisco, {C}alif., 1981), Contemp. Math.,
  vol.~20, Amer. Math. Soc., Providence, RI, 1983, pp.~305--325. \MR{MR718149
  (85e:57006)}

\bibitem{Menasco:Alternating}
\bysame, \emph{Closed incompressible surfaces in alternating knot and link
  complements}, Topology \textbf{23} (1984), no.~1, 37--44.

\bibitem{Milnor:free-action}
John Milnor, \emph{Groups which act on {$S^n$} without fixed points}, Amer. J.
  Math. \textbf{79} (1957), 623--630. \MR{0090056}

\bibitem{miyamoto}
Yosuke Miyamoto, \emph{Volumes of hyperbolic manifolds with geodesic boundary},
  Topology \textbf{33} (1994), no.~4, 613--629. \MR{1293303}

\bibitem{NeumannZagier}
Walter~D. Neumann and Don Zagier, \emph{Volumes of hyperbolic three-manifolds},
  Topology \textbf{24} (1985), no.~3, 307--332. \MR{815482}

\bibitem{Ozawa}
Makoto Ozawa, \emph{Non-triviality of generalized alternating knots}, J. Knot
  Theory Ramifications \textbf{15} (2006), no.~3, 351--360.

\bibitem{ozawa:adequate}
\bysame, \emph{Essential state surfaces for knots and links}, J. Aust. Math.
  Soc. \textbf{91} (2011), no.~3, 391--404.

\bibitem{Purcell:Cusps}
Jessica~S. Purcell, \emph{Cusp shapes under cone deformation}, J. Differential
  Geom. \textbf{80} (2008), no.~3, 453--500. \MR{2472480}

\bibitem{purcell:augmented}
\bysame, \emph{An introduction to fully augmented links}, Interactions between
  hyperbolic geometry, quantum topology and number theory, Contemp. Math., vol.
  541, Amer. Math. Soc., Providence, RI, 2011, pp.~205--220. \MR{2796634}

\bibitem{Purcell:book}
\bysame, \emph{Hyperbolic knot theory}, Available at
  \url{http://users.monash.edu/~jpurcell/hypknottheory.html}, 2017.

\bibitem{rivin:volume}
Igor Rivin, \emph{Euclidean structures on simplicial surfaces and hyperbolic
  volume}, Ann. of Math. (2) \textbf{139} (1994), no.~3, 553--580. \MR{1283870}

\bibitem{SakumaWeeks}
Makoto Sakuma and Jeffrey Weeks, \emph{Examples of canonical decompositions of
  hyperbolic link complements}, Japan. J. Math. (N.S.) \textbf{21} (1995),
  no.~2, 393--439. \MR{1364387}

\bibitem{thistle-tsviet}
Morwen Thistlethwaite and Anastasiia Tsvietkova, \emph{An alternative approach
  to hyperbolic structures on link complements}, Algebr. Geom. Topol.
  \textbf{14} (2014), no.~3, 1307--1337. \MR{3190595}

\bibitem{thurston:notes}
William~P. Thurston, \emph{The geometry and topology of three-manifolds},
  Princeton Univ. Math. Dept. Notes, 1979, Available at
  \url{http://www.msri.org/communications/books/gt3m}.

\bibitem{thurston:bulletin}
\bysame, \emph{Three-dimensional manifolds, {K}leinian groups and hyperbolic
  geometry}, Bull. Amer. Math. Soc. (N.S.) \textbf{6} (1982), no.~3, 357--381.

\bibitem{Thurston:Fiber}
\bysame, \emph{Hyperbolic structures on 3-manifolds {II}: surface groups and
  3-manifolds which fiber over the circle}, arXiv:math/9801045, 1998.

\bibitem{Tsvietkova}
Anastasiia Tsvietkova, \emph{Determining isotopy classes of crossing arcs in
  alternating links}, arXiv:1411.0231.

\bibitem{Weeks:Algorithm}
Jeff Weeks, \emph{Computation of hyperbolic structures in knot theory},
  Handbook of knot theory, Elsevier B. V., Amsterdam, 2005, pp.~461--480.
  \MR{2179268}

\bibitem{Worden:experiment}
William Worden, \emph{Experimental statistics of veering triangulations},
  Preprint available at \url{http://www.wtworden.org/Research/ESVT/}.

\bibitem{Yoshida:Smallest4Cusped}
Ken'ichi Yoshida, \emph{The minimal volume orientable hyperbolic 3-manifold
  with 4 cusps}, Pacific J. Math. \textbf{266} (2013), no.~2, 457--476.
  \MR{3130632}

\end{thebibliography}

\end{document}